\documentclass[12pt]{article}
\def\date{May 21, 2015}
\usepackage{amsmath, amssymb, amsfonts, theorem, color, enumitem,graphicx}

\newtheorem{proposition}{Proposition}[section]
\newtheorem{lemma}[proposition]{Lemma}
\newtheorem{corollary}[proposition]{Corollary}
\newtheorem{theorem}[proposition]{Theorem}
\newtheorem{claim}[proposition]{Claim}

\newcommand{\qed}{$\square$\bigskip}

\newenvironment{proof}{{\noindent\bf Proof. }}{\hfill \qed}
{\theorembodyfont{\rmfamily}
\newtheorem{definition}[proposition]{Definition}

}

\newcommand{\mcal}{\mathcal}

\newcommand{\C}{\mcal{C}}

\newcommand{\F}{\mcal{F}}

\def\xx#1{\oplus#1}

\begin{document}
\font\smallrm=cmr8

\phantom{a}\vskip .25in
\centerline{{\large \bf  FIVE-LIST-COLORING GRAPHS ON SURFACES II.}}
\smallskip
\centerline{{\large\bf A LINEAR BOUND FOR CRITICAL GRAPHS IN A DISK}}
\vskip.4in
\centerline{{\bf Luke Postle}%
\footnote{\texttt{lpostle@uwaterloo.ca}.}} 
\smallskip
\centerline{Department of Combinatorics and Optimization}
\centerline{University of Waterloo}
\centerline{Waterloo, ON}
\centerline{Canada N2L 3G1}
\medskip
\centerline{and}

\medskip
\centerline{{\bf Robin Thomas}%
\footnote{\texttt{thomas@math.gatech.edu}. Partially supported by NSF under
Grant No.~DMS-1202640.}}
\smallskip
\centerline{School of Mathematics}
\centerline{Georgia Institute of Technology}
\centerline{Atlanta, Georgia  30332-0160, USA}

\vskip 1in \centerline{\bf ABSTRACT}
\bigskip\bigskip

{
\noindent
Let $G$ be a plane graph with outer cycle $C$ and let $(L(v):v\in V(G))$ be 
a family of sets such that  $|L(v)|\ge5$ for every $v\in V(G)$.
By an $L$-coloring of a subgraph $J$ of $G$ we mean a (proper) coloring $\phi$ of $J$  such that $\phi(v)\in L(v)$
for every vertex $v$ of $J$.
We prove a conjecture of Dvo\v{r}\'ak et al.\ that  if $H$ is a minimal subgraph of $G$  such that
$C$ is a subgraph of $H$ and every $L$-coloring of $C$ that extends to an
$L$-coloring of $H$ also extends to an $L$-coloring of $G$, then $|V(H)|\le19|V(C)|$.

This is a lemma that plays an important role in subsequent papers,
because it motivates the study of graphs embedded in 
surfaces that satisfy an isoperimetric inequality suggested by this result.
Such study turned out to be quite profitable for the subject of list coloring graphs on surfaces.
}

\vfill \baselineskip 11pt \noindent September 2012. Revised \date.
\vfil\eject
\baselineskip 18pt

\section{Introduction}
\label{sec:intro}
All {\em graphs} in this paper are finite and simple; that is, they have no loops or parallel edges. 
{\em Paths} and {\em cycles} have no repeated vertices 
or edges. If $G$ is a graph and $L=(L(v):v\in V(G))$ is a family of sets, then we say that $L$ is a {\em list assignment for $G$}.
It is a {\em $k$-list-assignment}, if $|L(v)|\ge k$ for every vertex $v\in V(G)$.
An $L$-coloring  of $G$ is a (proper) coloring $\phi$ of $G$  such that $\phi(v)\in L(v)$
for every vertex $v$ of $G$.
We prove 
the following theorem which settles a conjecture of Dvo\v r\'ak et al.~\cite{AlbertsonsConj}.

\begin{theorem}\label{LinearCycle0}
Let $G$ be a plane graph with outer cycle $C$, let $L$ a $5$-list-assignment for $G$, and let $H$ be a minimal subgraph of
$G$ such that every $L$-coloring of $C$ that extends to an $L$-coloring of $H$ also extends to an $L$-coloring of $G$.
Then $H$ has at most $19|V(C)|$ vertices.
\end{theorem}

Earlier versions of this theorem were proved for ordinary coloring (that is, when all the lists in $L$ are equal)
by Thomassen~\cite[Theorem~5.5]{ThomCritical},
who proved it with $19|V(C)|$ replaced by $5^{|V(C)|^3}$, and by Yerger~\cite{YerPhD},
who improved the bound to $O(|V(C)|^3)$.
If every vertex of $G\setminus V(C)$ has degree at least five and all its neighbors in $G$ belong to $C$,
then the only graph $H$ satisfying the hypothesis of Theorem~\ref{LinearCycle0} is the graph $G$ itself.
It follows that our bound is asymptotically best possible.

The fact that our bound in Theorem~\ref{LinearCycle0} is linear in $|V(C)|$ turned out to be fairly significant.
In~\cite{PosThoHyperb} we define a family $\F$ of graphs embedded in surfaces to be {\em hyperbolic}
 if there exists a constant $c>0$
such that if $G\in\F$ is a graph that is embedded in a surface
$\Sigma$, then for every closed curve $\gamma:{\mathbb S}^1\to\Sigma$ that bounds
an open disk $\Delta$ and intersects $G$ only in vertices,
if $\Delta$ includes a vertex of $G$, then
the number of vertices of $G$ in $\Delta$ is at most
$c(|\{x\in {\mathbb S}^1\,:\, \gamma(x)\in V(G)\}|-1)$.
We say that $c$ is a \emph{Cheeger constant} for $\F$.
We were able to develop a structure theory of hyperbolic families, and that theory has applications to coloring,
as follows.

Let $L$ be a list assignment for a graph $G$.
We say that $G$ is {\em $L$-critica}l if $G$ is not $L$-colorable, but every proper subgraph of $G$ is.
It follows from Theorem~\ref{LinearCycle0} that the family of  embedded graphs 
that are $L$-critical for some $5$-list-assignment  $L$ is a hyperbolic.
The theory of hyperbolic families now implies that if
 $G$ is a graph embedded in a surface $\Sigma$ of genus $g$ and  $L$ is
a $5$-list-assignment for $G$, then
\begin{itemize}
\item if every non-null-homotopic cycle in $G$ has length $\Omega(\log g)$, then $G$ has an $L$-coloring, and
\item for every fixed $g$ there is a polynomial time algorithm to decide whether $G$ has an $L$-coloring.
\end{itemize}

Let us emphasize that the above  results are consequences of Theorem~\ref{LinearCycle0} and the
theory of hyperbolic families. Thus the same conclusion holds for other coloring problems that satisfy
the analog of Theorem~\ref{LinearCycle0} (with $19$ replaced by an arbitrary constant).
We will return to this in a moment.

The structure theory of  hyperbolic families suggests the following strengthening of  hyperbolicity.
Let $\F$ be a hyperbolic family of embedded graphs, let $c$ be a Cheeger constant for $\F$, and
let $d:=\lceil3(2c+1)\log_2(8c+4)\rceil$.
We say that $\F$ is {\em strongly hyperbolic} if there exists a constant $c_2$ such that 
for every $G\in\F$ embedded in a surface $\Sigma$ and for every two disjoint cycles $C_1,C_2$ of length at
most $2d$ in $G$, if there exists a cylinder $\Lambda\subseteq\Sigma$ with boundary components $C_1$ and $C_2$,
then $\Lambda$ includes at most $c_2$ vertices of $G$.

In a later paper of this series we will show that the family of embedded graphs that are $L$-critical for
some $5$-list-assignment $L$ is, in fact, strongly hyperbolic. Our theory of hyperbolicity then implies
that if
 $G$ is a graph embedded in a surface $\Sigma$ of genus $g$ and  $L$ is
a $5$-list-assignment for $G$, then
\begin{itemize}
\item if $G$ is $L$-critical, 
then $|V(G)|=O(g)$,
\item  if every non-null-homotopic cycle in $G$ has length $\Omega(g)$, and a set $X\subseteq V(G)$
of vertices that are pairwise at distance $\Omega(1)$ is precolored from the corresponding lists,
then the precoloring extends to an $L$-coloring of $G$, and
\item  if every non-null-homotopic cycle in $G$ has length $\Omega(g)$, and the graph $G$
is allowed to have crossings, but every two crossings are at distance $\Omega(1)$, then $G$ has
an $L$-coloring.
\end{itemize}

When combined with the strong hyperbolicity of another closely related family, we further obtain that
\begin{itemize}
\item if $G$ has at least one  $L$-coloring, then it has at least $2^{\Omega(|V(G)|)}$ distinct
$L$-colorings.
\end{itemize}

As indicated earlier, these results follow from the strong hyperbolicity of the family of $L$-critical
graphs, and hence the same results hold for other coloring problems as well.
The two other most interesting strongly hyperbolic families are the family of embedded graphs
of girth at least four that are $L$-critical for  some $4$-list-assignment $L$, and
 the family of embedded graphs
of girth at least five that are $L$-critical for  some $3$-list-assignment $L$.
We refer to~\cite{PosThoHyperb} for details.

In order to prove  Theorem~\ref{LinearCycle0} we prove a stronger version, 
stated below as Theorem~\ref{StrongLinear},
 which bounds the  number of vertices in terms of the sum of the sizes of large faces, a notion we call ``deficiency''. Another aspect to the proof is to incorporate the counting of neighbors of $C$ into the stronger formula. This allows the finding of reducible configurations close to the boundary in a manner similar to the discharging method's use of Euler's formula.

The paper is organized as follows.
In Section~\ref{subsec:criticalgraph} we define a more general notion of criticality for graphs and ``canvases'',
 which will be useful for proving Theorem~\ref{LinearCycle0}, and
we prove a structure theorem for said critical canvases. In Section~\ref{subsec:deficiency} we formally define deficiency and prove some lemmas about the deficiency of canvases. 
In Section~\ref{subsec:LinearCycle0} we formulate Theorem~\ref{StrongLinear} and prove several auxiliary results. 
In Section~\ref{sec:pfstronglin} we prove Theorem~\ref{StrongLinear}. 
In  Section~\ref{sec:conseqstronglin} we prove  Theorem~\ref{LinearCycle0}. 

\section{Critical Canvases}\label{subsec:criticalgraph}

In this section we define the notion of ``canvas", which will be used throughout the paper.
We define critical graphs and critical canvases, and prove several lemmas and 
Theorem~\ref{CycleChordTripod}, which may be regarded as a structure theorem for critical canvases.

In a plane graph $G$ exactly one of its faces is unbounded; we call that face the \emph{outer face} of $G$.
All other faces of $G$ are called \emph{internal}. If the outer face is a bounded by a cycle $C$, then we refer to $C$ as the \emph{outer cycle} of $G$.

We will need the following beautiful theorem of Thomassen~\cite{ThomPlanar}.
We state it in a slightly stronger form than~\cite{ThomPlanar}, but our version follows easily from the original by induction.
%
%

\begin{theorem}[Thomassen]\label{Thom2}
Let $G$ be a plane graph, $Z$ the set of vertices incident with the outer face of $G$, 
and let $S\subseteq Z$ be such that $|S|\le 2$, and if $|S|=2$ then the vertices in $S$ are adjacent.
Let $L$ be a list assignment for $G$ with $|L(v)|\ge 5$ for all $v\in V(G)\setminus V(Z)$,
$|L(v)|\ge 3$ for all $v\in V(Z)\setminus V(S)$,  $|L(v)|=1$ for all $v\in S$, and if $|S|=2$, then the lists of 
vertices in $S$ are disjoint. Then $G$ is $L$-colorable.
\end{theorem}

\begin{definition}[$T$-critical]
Let $G$ be a graph, $T \subseteq G$ a (not necessarily induced) subgraph of $G$ and $L$ a list assignment for $G$. 
For an $L$-coloring $\phi$ of $T$, we say that \emph{$\phi$ extends to an $L$-coloring of $G$} if there exists an $L$-coloring $\psi$ of $G$ such that $\phi(v)=\psi(v)$ for all $v\in V(T)$. The graph $G$ is \emph{$\phi$-critical} if $\phi$ extends to every proper subgraph of $G$ containing $T$ but not to $G$.

The graph $G$ is \emph{$T$-critical with respect to the list assignment $L$} if $G \ne T$ and for every proper subgraph $G' \subset G$ such that $T \subseteq G'$, there exists an $L$-coloring of $T$ that extends to an $L$-coloring of $G'$, but does not extend to an $L$-coloring of $G$. If the list assignment is clear from the context, we shorten this and say that $G$ is \emph{$T$-critical}.
\end{definition}

We need the following lemma about subgraphs of critical graphs.
If $G$ is a graph and $X\subseteq V(G)$, then let $G[X]$ denote the subgraph of $G$ induced by $X$.

\begin{lemma} \label{SComponent}
Let $T$ be a subgraph of a graph $G$ such that $G$ is $T$-critical with respect to a list assignment $L$. Let $A,B\subseteq G$ be such that $A\cup B=G$, $T\subseteq A$ and $B\ne A[V(A)\cap V(B)]$. Then $G[V(B)]$ is $A[V(A)\cap V(B)]$-critical.
\end{lemma}
\begin{proof}
Let $G'=G[V(B)]$ and $S=A[V(A)\cap V(B)]$. Since $G$ is $T$-critical, every isolated vertex of $G$ belongs to $T$, and
thus every isolated vertex of $G'$ belongs to $S$. Suppose for a contradiction
that $G'$ is not $S$-critical. Then, there exists an edge $e \in E(G') \setminus E(S)$ such
that every $L$-coloring of $S$ that extends to $G' \setminus e$ also extends to $G'$. Note
that $e \not\in E(T)$. Since $G$ is $T$-critical, there exists a coloring $\Phi$ of $T$ that
extends to an $L$-coloring $\phi$ of $G \setminus e$, but does not extend to an $L$-coloring
of $G$. However, by the choice of $e$, the restriction of $\phi$ to $S$ extends to an
$L$-coloring $\phi'$ of $G'$. Let $\phi''$ be the coloring that matches $\phi'$ on $V(G')$ and $\phi$
on $V(G) \setminus V(G')$. Observe that $\phi''$ is an $L$-coloring of $G$ extending $\Phi$, which
is a contradiction.
\end{proof}

\begin{definition}
\label{def:canvas}
We say the triple $(G,C,L)$ is a \emph{canvas} if $G$ is a $2$-connected plane graph, $C$ is its outer cycle, and $L$ is a list assignment for some graph $G'$ such that $G$ is a subgraph of $G'$,
$|L(v)|\ge 5$ for all $v\in V(G)-V(C)$ and there exists an $L$-coloring of $C$.
We say a canvas $(G,C,L)$ is \emph{critical} if $G$ is $C$-critical with respect to the list assignment~$L$.
\end{definition}

This definition of canvas differs from the one we used in~\cite{PosThoTwotwo} in two respects.
First, we allow $L$  to include lists of vertices that do not belong to $G$. That is just an artificial device
to make the notation easier when we pass to subgraphs.
Second, and more importantly, the current definition restricts the graph $G$ to be $2$-connected.
The reason for that is that frequently we will need to  manipulate faces of $G$ and doing so is much easier
when all the face boundaries are cycles.
That we can restrict to $2$-connected graphs follows from the next lemma.

\begin{lemma}\label{2Conn}
If $G$ is a plane graph, $C$ is its outer cycle, and $L$ is a list assignment for the vertices of $G$ such that
$G$ is $C$-critical with respect to~$L$, then $G$ is  $2$-connected, and hence $(G,C,L)$ is a canvas.
\end{lemma}

\begin{proof}
If $G$ is not $2$-connected, then it has subgraphs $A,B$ such that $A\cup B=G$, $|V(A\cap B)|\le 1$,
$C$ is a subgraph of $A$ and $V(B)-V(A)\ne\emptyset$.
By Lemma~\ref{SComponent} the graph $G[V(B)]$ is $A[V(A)\cap V(B)]$-critical, contrary to
Theorem~\ref{Thom2}.
\end{proof}



Lemma~\ref{SComponent} has a useful corollary. To state it, however, we need 
notation for a subgraph of a plane graph $G$, where the subgraph consists of vertices and edges drawn in the closed disk bounded by a cycle $C$ of $G$. 
In fact, we will need this notation even when $C$ uses edges that do not belong to $G$.
Hence the following definition.

\begin{definition}
 Let $T=(G,C,L)$ be a canvas, and let $G'$ be a plane graph obtained from
 $G$ by adding a (possibly empty) set of edges inside internal faces of $G$.
 If $C'$ is a cycle in $G'$, we let $G\langle C' \rangle$ denote the subgraph of $G\cup C'$ contained in the closed disk bounded by $C'$. We let $T\langle C' \rangle$ denote the canvas $(G \langle C'\rangle, C', L)$.
 \end{definition}

\begin{corollary} \label{SubCycle}
Let $T=(G,C,L)$ be a critical canvas. If $C'$ is a cycle in $G$ such that $G\langle C' \rangle \ne C'$, then $T\langle C' \rangle$ is a critical canvas.
\end{corollary}
\begin{proof}
Let $B=G\langle C' \rangle$ and let $A$ be obtained from $G$ by deleting all vertices and edges drawn in the open disk bounded by $C'$. By applying Lemma~\ref{SComponent}, it follows that $G\langle C' \rangle$ is $A[V(C')]$-critical,
and hence $C'$-critical. By Lemma~\ref{2Conn}, $T\langle C' \rangle$ is a critical canvas.
\end{proof}

\begin{definition}
Let $T=(G,C,L)$ be a canvas and $G'\subseteq G$ such that $C\subseteq G'$ and $G'$ is $2$-connected. We define the \emph{subcanvas} of $T$ induced by $G'$ to be $(G',C,L)$ and we denote it by $T[G']$.
\end{definition}

Another useful fact is the following. We omit the proof, which is easy.

\begin{proposition}
\label{p:critsubcanvas}
Let $T=(G,C,L)$ be a canvas such that there exists a proper $L$-coloring of $C$ that does not extend to $G$. Then $T$ contains a critical subcanvas.
\end{proposition}

If $G$ is a graph, we let $N(v)$ denote the set of neighbors of a vertex $v$ of $G$. If $X\subseteq V(G)$, we let $N(X)$ denote the set of vertices of $G$ not in $X$ with at least one neighbor in $X$. If $C$ is a cycle in a graph $G$, then a \emph{chord} of $C$ is an edge $e$ in $E(G)\setminus E(C)$ with both ends in $V(C)$.
The following theorem gives useful information about the structure of critical canvases.

\begin{theorem}\label{CycleChordTripod} (Cycle Chord or Tripod Theorem)\\
If $T=(G,C,L)$ is a critical canvas, then either 
\begin{enumerate}
\item $C$ has a chord in $G$, or
\item there exists a vertex $v\in V(G)\setminus V(C)$ with at least three neighbors on $C$ 
such that at most one of the faces of $G[\{v\}\cup V(C)]$ includes a vertex or edge of $G$.
\end{enumerate}

\end{theorem}
\begin{proof}
Suppose $C$ does not have a chord. Let $X$ be the set of vertices with at least three neighbors on $C$. Let $G'$ be the subgraph of $G$ defined by $V(G')=C\cup X$ and $E(G')=E(G[C\cup X])-E(G[X])$. 

We claim that if $f$ is face of $G'$ such that $f$ is incident with at most one vertex of $X$, then $f$ does not include a vertex or edge of $G$. Suppose not. Let $C'$ be the boundary of $f$. As $C$ has no chords and every edge with one end in $X$ and the other in $C$ is in $E(G')$, it follows that $C'$ has no chords. As $T$ is critical, there exists an $L$-coloring $\phi$ of $G\setminus (V(G\langle C' \rangle) \setminus V(C'))$ which does not extend to $G$. Hence, the restriction of $\phi$ to $C'$ does not extend to $V(G\langle C' \rangle$). Let $G''= G\langle C'\rangle \setminus V(C)$, $S=V(C')\setminus V(C)$, $L'(v)=\{\phi(v)\}$ for $v\in S$ and $L'(v)=L(v)\setminus \{\phi(x): x\in V(C)\cap N(v)\}$ for $v\in V(G'')\setminus S$. Note that $|L'(v)|\ge 3$ for all $v\not \in S$ by definition of $X$. By Theorem~\ref{Thom2}, there exists an $L'$-coloring of $G''$ and hence an $L$-coloring of $G$ which extends $\phi$, a contradiction. This proves the claim.

As $T$ is critical, $G\ne C$. As $C$ has no chords, it follows from the claim above that $X\ne \emptyset$. Let $\F$ be the set of internal faces of $G'$ incident with at least two elements of $X$. Consider the graph $H$ whose vertices are $X\cup \F$, where a vertex $x\in X$ is adjacent to $f\in \F$ if $x$ is incident with $f$. By planarity, $H$ is a tree. Let $v$ be a leaf of $H$. By the definition of $H$, $v\in X$. Hence at most one of the faces of $G[\{v\}\cup V(C)]$ is incident with another vertex of $X$. Yet all other faces of $G[\{v\}\cup V(C)]$ are incident with only one element of $X$, namely $v$, and so by the claim above, these faces do not include a vertex or edge of $C$, as desired.
\end{proof}

We also need the following proposition, where the first  statement is a consequence of Theorem~\ref{Thom2} and the second follows directly.

\begin{proposition}\label{Facts}
If $T=(G,C,L)$ is a critical canvas, then
\begin{enumerate}
\item[{\rm(1)}] for every cycle $C'$ of $G$ of length at most four, $V(G\langle C'\rangle) =V( C')$, and
\item[{\rm(2)}] every vertex in $V(G)\setminus V(C)$ has degree at least five.
\end{enumerate}
\end{proposition}

\section{Deficiency}\label{subsec:deficiency}

In this section we introduce the notion of deficiency, which will play a pivotal role in the rest of the paper,
 and we prove several basic lemmas about deficiency.

\begin{definition}
Let $G$ be a plane graph with outer cycle $C$.
We say that a vertex $v\in V(G)$ is {\em internal} if $v\not\in V(C)$.
 We denote the number of internal vertices by $v(G)$, and we
define the \emph{deficiency} of $G$, denoted by ${\rm def}(G)$, as
$${\rm def}(G):=|E(G)\setminus E(C)|-3v(G).$$
If $T=(G,C,L)$ is a canvas, then we define $v(T):=v(G)$ and ${\rm def}(T):={\rm def}(G)$.
\end{definition}

%


\begin{definition}
Let $G$ be  a $2$-connected plane graph.
We let $\F(G)$ denote the set of internal faces of $G$. 
If $f$ is a face of $G$, then we let $|f|$ denote the length of the cycle bounding $f$.
Likewise, if $C$ is a cycle in $G$, then we denote its length by $|C|$.
For $f\in\F(G)$ we let $C_f$ be the cycle bounding $f$. We denote by $G[f]$ the subgraph $G\langle C_f \rangle$. If $T=(G,C,L)$ is a canvas and $f$ is a face of $G$, let $T[f]$ denote the canvas $T\langle C_f \rangle$, that is, $(G[f],C_f, L)$.
\end{definition}

The following is an equivalent formula for the deficiency of a $2$-connected plane graph.

\begin{lemma} \label{deffaces}
If $G$ is a $2$-connected plane graph with outer cycle $C$, then
$${\rm def}(G)=|C|-3 - \sum_{f\in \F(G)} (|f|-3).$$
\end{lemma}  

\begin{proof}
Euler's formula gives $|C|+v(G)+|\F(G)|+1=|E(G)|+2$, and hence

\begin{centering}
$|C|-3 - \sum_{f\in \F(G)} (|f|-3)=|C|-3-(2|E(G)|-|C|)+3|\F(G)|=$\\
$2|C|-2|E(G)|-3+3|E(G)|-3|C|-3v(G)+3=$\\
$|E(G)\setminus E(C)|-3v(G)={\rm def}(G)$,\\
\end{centering}
\noindent 
as desired.
\end{proof}

\noindent 
In fact, it was the above formula that led us to the notion of deficiency. However, for most of the proof
it is more convenient to use the definition of deficiency. We will not need Lemma~\ref{deffaces} until
the proof of Theorem~\ref{LinearCycle0} in Section~\ref{sec:conseqstronglin}.

%
%

\begin{lemma} \label{defsum}
If $G$ is a $2$-connected plane graph with outer cycle $C$ and $G'$ is a $2$-connected subgraph of $G$ 
containing $C$, then
$${\rm def}(G) = {\rm def}(G') + \sum_{f\in \F(G')} {\rm def}(G[f]).$$ 
\end{lemma}  

\begin{proof}
The lemma  follows from the fact that every internal vertex of $G$ 
 is an internal vertex of exactly one of the graphs
$G'$ and $G[f]$ for $f\in \F(G')$, and the same holds for edges not incident with the outer face.
\end{proof}

\begin{theorem}\label{CycleDef} (Cycle Sum of Faces Theorem)\\
If $T=(G,C,L)$ is a critical canvas, then ${\rm def}(T)\ge 1$.
\end{theorem}
\begin{proof}
We proceed by induction on the number of vertices of $G$. We apply Theorem~\ref{CycleChordTripod} to~$T$.
Suppose first that (1) holds; that is, there is a chord $e$ of $C$. Let $C_1,C_2$ be the cycles of $C+e$ other than $C$. Hence $|V(C_1)|+|V(C_2)|=|V(C)|+2$. Let $T_1=T\langle C_1 \rangle=(G_1,C_1,L)$ and $T_2=T\langle C_2 \rangle=(G_2,C_2,L)$. By Lemma~\ref{defsum} applied to $G'=C+e$, ${\rm def}(T)= {\rm def}(T_1)+{\rm def}(T_2)+1$.
By Corollary~\ref{SubCycle}, for $i\in\{1,2\}$, either $T_i$ is critical or $G_i=C_i$. If $G_i\ne C_i$, then ${\rm def}(T_i)\ge 1$ by induction. If $G_i = C_i$, then ${\rm def}(T_i)=0$ by definition. In either case, ${\rm def}(T_i)\ge 0$. Thus ${\rm def}(T)\ge 0+0+1\ge 1$, as desired.

So we may suppose that (2) holds; that is, there exists an internal vertex $v$ of $G$ 
such that $v$ is adjacent to at least three vertices of $C$ and at most one of the faces of $G[\{v\} \cup V(C)]$ includes a vertex or edge of $G$. 
Let $G'=G[\{v\}\cup C]$. First suppose that none of the faces of $G'$ includes a vertex or edge of $G$, and hence $V(G)=V(C)\cup \{v\}$. As $G$ is $C$-critical, it follows from Proposition~\ref{Facts}(2) that $v$ must have degree at least five. Thus, ${\rm def}(T)\ge 2$, as desired.

So we may suppose that exactly one of the faces of $G'$ includes a vertex or edge of $G$. 
Let $C'$ be the boundary of that face. 
We have 
$${\rm def}(T)\ge {\rm def}(T\langle C'\rangle)\ge1,$$
where the first inequality follows from the definition of deficiency and the second by induction, because
$T \langle C' \rangle$ is a critical canvas by Corollary~\ref{SubCycle}.
\end{proof}


To handle critical canvases with at most four internal vertices will need the following inequality.

\begin{lemma} \label{defbound}
If $G$ is a $2$-connected plane graph with outer cycle $C$ and every internal vertex of $G$
has degree at least five, then
$${\rm def}(G) \ge 2v(G)-|E(G\setminus V(C))|,$$
with equality if and only if every vertex  of $G$
has degree exactly five.
\end{lemma}  

\begin{proof}
By Proposition~\ref{Facts}(2)
$${\rm def}(G) =|E(G)\setminus E(C)|-3v(G)\ge 5v(G)- |E(G\setminus V(C))|-3v(G)=2v(G)-|E(G\setminus V(C))|,$$
with equality if and only if every vertex  of $G$
has degree exactly five.
\end{proof}

\section{Linear Bound for Cycles}\label{subsec:LinearCycle0}

The purpose of this section is to state Theorem~\ref{StrongLinear}, the desired strengthening of Theorem~\ref{LinearCycle0}.
First a few definitions. If $G$ is a plane graph and $u,v\in V(G)$, then we say that $u$ and $v$ are \emph{cofacial} if there exists a face $f$ of $G$ such that $u$ and $v$ are both incident with $f$.

\begin{definition}
Let $G$  be a $2$-connected plane graph with outer cycle $C$.
We define the \emph{boundary} of $G$, denoted by $B(G)$, as $N(V(G))$. 
We also define the \emph{quasi-boundary} of $G$, denoted by $Q(G)$, as the set of vertices not in $C$ that are cofacial 
with at least one vertex of~$C$. We let $b(G) := |B(G)|$ and $q(G) := |Q(G)|$. Note that $B(G)\subseteq Q(G)$.

If  $T=(G,C,L)$ is a canvas, then we extend the above notions to $T$ in the obvious way,
so that we can speak of the boundary or quasi-boundary of $T$, and we define $B(T):=B(G)$
and similarly for all the other quantities.
\end{definition}

For the rest of this paper let $\epsilon,\alpha > 0$ be fixed positive real numbers.
Our main result, Theorem~\ref{StrongLinear}, depends on $\epsilon$ and $\alpha$ and holds as long as 
 $\epsilon$ and $\alpha$  satisfy three natural inequalities. 
Later we will make a specific choice of $\epsilon$ and $\alpha$ in order to optimize the
constant in Theorem~\ref{LinearCycle0}. We need to introduce the following quantities.

\begin{definition}
Let $G$  be a $2$-connected plane graph. We define
 $$s(G) := \epsilon v(G) + \alpha(b(G)+q(G))\hbox{\quad and \quad}d(G) := {\rm def}(G)-s(G).$$
If  $T=(G,C,L)$ is a canvas, then we define $s(T) :=s(G)$ and $d(T) :=d(G)$.
\end{definition}

We need to establish a few properties of the quantities just introduced before we can state Theorem~\ref{StrongLinear}.

\begin{proposition}\label{surplussum}
Let $G$  be a $2$-connected plane graph with outer cycle $C$, and let $G'$ be a $2$-connected
subgraph of $G$ containing $C$ as a subgraph. Then
\begin{itemize}
\item $v(G)=v(G')+\sum_{f\in\F(G')} v(G[f])$,
\item $b(G)\le b(G')+\sum_{f\in\F(G')} b(G[f])$,
\item $q(G)\le q(G')+\sum_{f\in\F(G')} q(G[f])$,
\item $s(G)\le s(G')+\sum_{f\in\F(G')} s(G[f])$,
\item $d(G)\ge d(G')+\sum_{f\in\F(G')} d(G[f])$.
\end{itemize}
\end{proposition}
\begin{proof}
For $f\in \F(G')$, let $C_f$ denote the cycle bounding $f$.
The first assertion follows as every vertex of $V(G)\setminus V(C)$ is in exactly one of 
$G'\setminus V(C)$ and $G[f]\setminus V(C_f)$ for $f\in \F(G')$, and every vertex in one of those sets is in $V(G)\setminus V(C)$.

The second assertion follows from the claim that $B(G)\subseteq B(G')\cup \bigcup_{f\in\F(G')} B(G[f])$. 
To see this claim, suppose that $v\in B(G)$. Now $v\in B(G)$ if and only if $v$ has a neighbor $u$ in $V(C)$. If $v\in V(G')$, then $v\in B(G')$. So we may assume that $v$ is a vertex of $G[f]\setminus V(C_f)$ for some $f\in \F(G')$. So it must be that $u\in V(C_f)$ and hence $v\in B(G[f])$.

The third assertion follows from the claim that $Q(T)\subseteq Q(G')\cup \bigcup_{f\in\F(G')} Q(G[f])$. That claim follows with the same argument as above, except that $u\in V(C)$ is cofacial with $v$ instead of a neighbor of $v$.

The fourth statement follows from the first three. The fifth statement follows from the fourth and Lemma~\ref{defsum}. 
\end{proof} 

\begin{corollary}\label{ChordOr2Path}
Let $G$  be a $2$-connected plane graph with outer cycle $C$.
If $e$ is a chord of $C$ and $C_1,C_2$ are the cycles of $C+e$ other than $C$, then 
$$d(G)\ge d(G\langle C_1 \rangle )+d(G\langle C_2 \rangle)+1.$$
If $v$ is a vertex with two neighbors $u_1,u_2\in V(C)$ and $C_1,C_2$ cycles such that $C_1\cap C_2=u_1vu_2$ 
and $C_1\cup C_2=C\cup u_1vu_2$, then
$$d(G)\ge d(G\langle C_1 \rangle)+d(G\langle C_2 \rangle)-1-(2\alpha+\epsilon).$$
\end{corollary}
\begin{proof}
Both formulas follow from Proposition~\ref{surplussum} applied to $G':=G[C_1\cup C_2]$.
\end{proof}

For future convenience we state the following facts, which follow directly from the definitions.

\begin{proposition} \label{d0}
Let $T=(G,C,L)$ be a canvas.
\begin{enumerate}
\item[{\rm(i)}] If $v(T)=0$, then $d(T)= |E(G)\setminus E(C)|$.
\item[{\rm(ii)}] If $v(T)=1$, then $d(T)= |E(G)\setminus E(C)|-3 - (2\alpha+\epsilon)$.
\end{enumerate}
\end{proposition}
%
%


\medskip
We are now ready to state our generalization of Theorem~\ref{LinearCycle0}. 

\begin{theorem}\label{StrongLinear}
Let $\epsilon, \alpha, \gamma > 0$ satisfy the following:
\begin{enumerate}
\item[{\rm(I1)}] $3\epsilon \le 2\alpha$,
\item[{\rm(I2)}]  $6\alpha+3\epsilon\le \gamma$,
\item[{\rm(I3)}]  $2\alpha+3\epsilon+\gamma \le 1$.
\end{enumerate}
If $T=(G,C,L)$ is a critical canvas and $v(T)\ge 2$, then $d(T)\ge 3-\gamma$.
\end{theorem}

\section{Proof of Theorem~\ref{StrongLinear}}
\label{sec:pfstronglin}

This section is devoted to a proof of Theorem~\ref{StrongLinear}. We proceed in a series of claims. Throughout this section let $T=(G,C,L)$ be a counterexample to Theorem~\ref{StrongLinear} such that 
\begin{itemize}
\item[(M1)] $|E(G)|$ is minimum,
\end{itemize}
and, subject to that,
\begin{itemize}
\item[(M2)] $\sum_{v\in V(G)}|L(v)|$ is minimum.
\end{itemize}
Let us recall the useful facts of Proposition~\ref{Facts}, especially that there is no triangle $C'$ of $G$ with 
$G\langle C'\rangle \ne C'$ and that $\deg(v)\ge 5$ for all $v\in V(G)\setminus V(C)$.

\begin{claim}\label{Base}
$v(T)\ge 5.$
\end{claim}
\begin{proof}
Suppose for a contradiction that $v(T)\le4$, and let $m:=|E( G\setminus V(C))|$.
Then $m\le 1$ if $v(G)=2$, $m\le 3$ if $v(G)=3$,  and $m\le 5$ if $v(G)=4$, where the last inequality
follows from Proposition~\ref{Facts}(1).
By Proposition~\ref{Facts}(2) and Lemma~\ref{defbound} we have def$(G)\ge3$. 
Furthermore,  if  $v(T)=4$, then  equality holds if and only if  $m=5$ and every internal
vertex of $G$ has degree exactly five. 
However, that cannot happen, because in that case $G\setminus V(C)$ is obtained
from the complete graph on four vertices by deleting an edge, and hence every $L$-coloring
of $C$ extends to an $L$-coloring of $G$, contrary to the criticality of $T$.
Thus def$(G)\ge4$ when $v(T)=4$.
Clearly $s(G)\le v(G)(2\alpha+\epsilon)$, and hence
$d(G)\ge 3-3(2\alpha+\epsilon)\ge3-\gamma$ by inequality (I2)  when $v(G)\in\{2,3\}$,
and $d(G)\ge 4-4(2\alpha+\epsilon)\ge3-\gamma$ by inequalities  (I2) and (I3) when $v(G)=4$,
in either case contrary to the fact that $T$ is a counterexample to Theorem~\ref{StrongLinear}.
\end{proof}

\subsection{Proper Critical Subgraphs}

Here is a remarkably useful lemma.

\begin{claim}\label{ProperCrit}
Suppose $T_0=(G_0,C_0,L_0)$ is a critical canvas with $|E(G_0)|\le |E(G)|$ and $v(T_0)\ge 2$,
and let $G'$ be a proper subgraph of $G_0$ such that for some list assignment $L'$ the triple
$(G',C_0,L')$  is a critical canvas. Then

\begin{enumerate}
\item[{\rm(1)}] $d(T_0)\ge 4-\gamma$, and
\item[{\rm(2)}] $d(T_0)\ge 4-2(2\alpha+\epsilon)$ if $|E(G_0)\setminus E(G')|,|E(G')\setminus E(C_0)|\ge 2$, and
\item[{\rm(3)}] $d(T_0)\ge 5-2\alpha-\epsilon - \gamma$ if $|E(G_0)\setminus E(G')|,|E(G')\setminus E(C_0)|\ge 2$ and $v(T_0)\ge 3$.
\end{enumerate}
\end{claim}

\begin{proof}
Note that the inequality in (3) implies the inequality in (2) implies the inequality in (1) by inequalities (I2) and (I3). 
Given Proposition~\ref{d0} and the fact that $T$ is a minimum counterexample, it follows that $d(T_0[f])\ge 0$ for all $f\in \F(G')$ and $d(T_0[f])\ge 1$ if $f$ includes a vertex or edge of $G_0$. Moreover as $G'$ is a proper subgraph
of $G_0$, there exists at least one $f\in \F(G')$ such that $f$ includes a vertex or edge of $G_0$. Furthermore, if $|E(G_0)\setminus E(G')|\ge 2$, either there exist two such $f$'s or $d(T_0[f])\ge 2-(2\alpha+\epsilon)$ for some $f\in \F(G')$ by Proposition~\ref{d0}. 

Now $d(T_0)\ge d(G')+\sum_{f\in\F(G')}d(T_0[f])$ by Proposition~\ref{surplussum}. As noted above though, $\sum_{f\in\F(G')}d(T_0[f])\ge 1$ and is at least $2-(2\alpha+\epsilon)$ if $|E(G_0)\setminus E(G')|\ge 2$. Furthermore if $v(T_0[f])\ge 2$, then $d(T_0[f])\ge 3-\gamma$.

Assume first that $v(G')>1$. Then $d(G')\ge 3-\gamma$ as $T$ is a minimum counterexample. Hence $d(T_0)\ge 4-\gamma$ if $|E(G_0)\setminus E(G')|=1$ and (1) holds as desired. 
Otherwise $d(T_0)\ge 5-(2\alpha+\epsilon) - \gamma$ and (2) and (3) hold as desired.

So we may assume that $v(G')\le 1$. Suppose $v(G')=1$. 
Thus $d(G')\ge 2-(2\alpha+\epsilon)$ by Proposition~\ref{d0} and criticality. Moreover, there exists $f\in \F(G')$ such that $v(T_0[f])\ge 1$. If $v(T_0[f])\ge 2$, then $d(T_0[f])\ge 3-\gamma$ as $T$ is a minimum counterexample. 
As above, $d(T_0)\ge d(G')+d(T_0[f])\ge 5-(2\alpha+\epsilon)-\gamma$ and (3) holds, as desired. If $v(T_0[f])=1$, then $d(T_0[f])\ge 2-(2\alpha+\epsilon)$ by Proposition~\ref{d0}. 
As above, $d(T_0)\ge d(G')+d(T[f])\ge 2(2-(2\alpha+\epsilon))=4-2(2\alpha+\epsilon)$ and (1)
 and (2) hold, as desired. Yet if $v(T_0)\ge 3$, there must be two such faces if no face has at least two internal vertices. In that case then, $d(T_0)\ge 3(2-(2\alpha+\epsilon)) = 6 - 3(2\alpha+\epsilon)$ which is at least $5-(2\alpha+\epsilon)-\gamma$ by inequality (I2) and (3) holds, as desired.

So suppose $v(G')= 0$. As $G'\ne C_0$, $d(T_0[G'])\ge |E(G')\setminus E(C_0)|$ by Proposition~\ref{d0}.  As $v(T_0)\ge 2$, either there exists $f\in\F(G')$ such that $v(T_0[f])\ge 2$ or there exist $f_1,f_2\in \F(G')$ such that $v(T_0[f_1]),v(T_0[f_2])\ge 1$. Suppose the first case. Then $d(T_0[f])\ge 3-\gamma$ as $T$ is a minimum counterexample. Hence $d(T_0)\ge |E(G')\setminus E(C_0)|+3-\gamma$. Since $|E(G')\setminus E(C_0)|\ge 1$, $d(T_0)\ge 4-\gamma$ and (1) holds, as desired. If $|E(G')\setminus E(C_0)|\ge 2$, then $d(T_0)\ge 5-\gamma$ and (2) and (3) hold, as desired. So suppose the latter. Then $d(T_0[f_1]),d(T_0[f_2])\ge 2-(2\alpha+\epsilon)$ and $d(T_0)\ge 1+2(2-(2\alpha+\epsilon)) = 5-2(2\alpha+\epsilon)$, and all three statements hold as desired, since $2\alpha+\epsilon\le \gamma$ by inequality (I2).
\end{proof}

\begin{claim}\label{NoProperCrit}
There does not exist a proper $C$-critical subgraph $G'$ of $G$.
\end{claim}
\begin{proof}
This follows from Claim~\ref{ProperCrit} applied to $T_0=T$.
\end{proof}

\begin{claim}\label{Chord}
There does not exist a chord of $C$.
\end{claim}
\begin{proof}
Suppose there exists a chord $e$ of $C$. Let $G'=C+e$. As $v(T)\ne 0$, $G'$ is a proper subgraph of $G$. Yet $G'$ is $C$-critical, contradicting Claim~\ref{NoProperCrit}.  
\end{proof}

\subsection{Dividing Vertices}

\begin{definition}
Let $G_0$ be a $2$-connected plane graph with outer cycle $C_0$.
Let $v$ be an internal vertex of $G_0$ and suppose there exist two distinct faces $f_1,f_2\in \F(G_0)$ such that for $i\in \{1,2\}$ the boundary of $f_i$ includes $v$ and a vertex of $C_0$, say $u_i$. 
Let us assume that $u_1\ne u_2$ and let
$G'$ be the plane graph obtained from $G_0$ by adding the edges $u_1v, u_2v$ if they are not present in $G_0$. 
Consider the cycles $C_1,C_2$ of $G'$, where $C_1\cap C_2 =u_1vu_2$ and $C_1\cup C_2 = C_0\cup u_1vu_2$. 
If for both $i\in \{1,2\}$ we have $|E(T\langle C_i\rangle)\setminus E(C_i)|\ge 2$, then we say that $v$ is a \emph{dividing} vertex of $G_0$. 
If for both $i\in\{1,2\}$ we have $v(T\langle C_i \rangle)\ge 1$, we say $v$ is a \emph{strong dividing} vertex of $G_0$. 
If $v$ is a dividing vertex of $G_0$ and the edges $u_1v, u_2v$ belong to $G_0$, then we say that $v$ is a \emph{true dividing} vertex of $G_0$.
If  $T_0=(G_0,C_0,L_0)$ is a canvas, then by a \emph{dividing} vertex of $T_0$ we mean a
dividing vertex of $G_0$, and similarly for  strong and true dividing vertices.
\end{definition}

 
\begin{claim}\label{DividingTrue}
Suppose $T_0=(G_0,C_0,L_0)$ is a critical canvas with $|E(G_0)|\le |E(G)|$ and $v(T_0)\ge 2$. If $G_0$ contains a true dividing vertex, then 

\begin{enumerate}
\item[{\rm(1)}] $d(T_0)\ge 3-2(2\alpha+\epsilon)$, and
\item[{\rm(2)}] $d(T_0)\ge 4-2\alpha-\epsilon-\gamma$ if $v(T_0)\ge 3$. 
\end{enumerate}
\end{claim}
\begin{proof}
Note that the inequality in (2) implies the inequality in (1) by inequality (I3). Let $u_1,u_2, C_1, C_2$ and $v$ be as in the definition of true dividing vertex. Since $v$ is true, $u_1v$ and $u_2v$ belong to $G_0$. Let $G'=C_1\cup C_2$. Hence $G', C_1$ and $C_2$ are subgraphs of $G_0$. Thus $d(T_0[G'])=-1-(2\alpha+\epsilon)$ by Proposition~\ref{d0}(ii).

Note that, by Corollary~\ref{SubCycle}, $T_0\langle C_1 \rangle, T_0 \langle C_2\rangle$ are critical canvases. If $v(T_0\langle C_1 \rangle)=0$, then $d(T_0\langle C_1 \rangle)\ge 2$ by Proposition~\ref{d0}(i), because $|E(T_0[C_1])\setminus E(C_1)|\ge 2$ by the definition of dividing vertex. If $v(T_0\langle C_1 \rangle)=1$, then $d(T_0\langle C_1 \rangle)\ge 2-(2\alpha+\epsilon)$ by Proposition~\ref{d0}(ii). If $v(T_0\langle C_1 \rangle)\ge 2$, then $d(T_0\langle C_1 \rangle)\ge 3-\gamma$ as $T$ is a minimum counterexample. In any case, $d(T_0\langle C_1\rangle)\ge 2-(2\alpha+\epsilon)$ as $\gamma\le 1+(2\alpha+\epsilon)$ by inequality (I3). Similarly, $d(T_0[C_2])\ge 2-(2\alpha+\epsilon)$. 

By Proposition~\ref{surplussum}, $d(T_0)\ge d(T_0[G'])+d(T_0\langle C_1\rangle)+d(T_0\langle C_2\rangle)$. Now let us choose $v$ such that $a=\min \{v(T_0\langle C_1 \rangle), v(T_0\langle C_2 \rangle)\}$ is minimized. 
Note then that $a\ne 1$, as otherwise there exists another true dividing vertex,
 contradicting the minimality of $a$. First suppose that $a\ge 2$ and hence 

$$d(T_0) \ge (-1-(2\alpha+\epsilon)) + 2(3-\gamma) = 5-2\gamma - (2\alpha+\epsilon)$$

\noindent and (1) and (2) hold by inequality (I3), as desired.  So we may suppose that $a=0$. Hence 

$$d(T_0) \ge (-1-(2\alpha+\epsilon)) + 2 + 2-(2\alpha+\epsilon) = 3-2(2\alpha+\epsilon)$$

\noindent and (1) holds, as desired. Yet if $v(T_0)\ge 3$, then 

$$d(T_0) \ge (-1-(2\alpha+\epsilon)) + 2 + 3-\gamma = 4-\gamma-(2\alpha+\epsilon)$$

\noindent and (2) holds, as desired.
\end{proof}

\begin{claim}\label{DividingStrong}
Suppose $T_0=(G_0,C_0,L_0)$ is a critical canvas with $|E(G_0)|\le |E(G)|$. If $G_0$ contains a strong dividing vertex, then $d(T_0)\ge 4-2\gamma$.
\end{claim}
\begin{proof}
Let $u_1,u_2, C_1, C_2$ and $v$ be as in the definition of strong dividing vertex. As $v$ is strong, $v(T_0)\ge 3$. Note that $v(T_0\langle C_1\rangle)\ge 1$ as $v$ is strong.
 If $v(T_0\langle C_1 \rangle)=1$, then the unique vertex in $V(T_0\langle C_1 \rangle)\setminus V(C_1)$ 
is a true dividing vertex and hence
$$d(T_0)\ge 4-2\alpha-\epsilon-\gamma\ge 4-2\gamma$$
by Claim~\ref{DividingTrue} and inequality (I2), as desired.
So we may assume that $v(T_0\langle C_1 \rangle)\ge 2$ and similarly that $v(T_0\langle C_2\rangle)\ge 2$.





Let $G_0'$ be the graph obtained from $G_0$ by adding vertices $z_1,z_2$ and edges $u_1z_1,z_1v,vz_2,z_2u_2$. 
Similarly let $G'$ be the graph obtained from $C_0$ by adding vertices $v,z_1,z_2$ and edges 
$u_1z_1$, $z_1v$, $vz_2$, $z_2u_2$.
 Let $L_0'(x)=L_0(x)$ for all $x\in V(G_0)$ and $L_0(z_1)=L(z_2)=R$, where $R$ is a set of five new colors,
and let $T_0'=(G_0',C_0,L_0')$.
Now 
$${\rm def}(G')=|E(G')|-|E(C_0)|-3v(G')=4-3\cdot3=-5.$$

Since $G_0$ is $C_0$-critical,  there exists a coloring $\phi_0$ of $C_0$ that  does not extend to $G_0$.
By Claim~\ref{ProperCrit} the graph $G_0$ does not have a proper $C_0$-critical subgraph, and hence
$\phi_0$ extends to every proper subgraph of $G_0$ by Proposition~\ref{p:critsubcanvas}.
%
For every $c\in L(v)$, let $\phi_c(v)=c$, $\phi_c(z_1)=\phi_c(z_2)\in R$ and $\phi_c(x)=\phi_0(x)$ for all $x\in C_0$. 
Let $C_1',C_2'$ be the two facial cycles of $G'$ other than $C_0$. 
Since $\phi_0$ does not extend to $G_0$, for every $c\in L(v)$ the coloring  $\phi_c$ does not extend to an $L$-coloring of either $G_0'\langle C_1' \rangle$ or $G_0'\langle C_2' \rangle$. 
Since $|L(v)|\ge 5$
there exists $i\in\{1,2\}$ such that there exist at least three colors $c$ in $L(v)$ such that $\phi_c$ does not extend to an $L$-coloring of $G_0'\langle C_i'\rangle$. We may assume without loss of generality that $i=1$. 
Let $\C$ be the set of all colors $c\in L(v)$ such that $\phi_c$ does not extend to an $L$-coloring of $G_0' \langle C_1' \rangle$.
Thus $|\C|\ge3$.

Let $G_1'$ be the graph obtained from $G_0'\langle C_1'\rangle$ by adding the edge $z_1z_2$
inside the outer face of $G_0'\langle C_1'\rangle$.
Let $C_1''=(C_1'\setminus v)+z_1z_2$. Let $L'(z_1)=\{c_1\}, L'(z_2)=\{c_2\}$ where $c_1,c_2$ are two distinct entirely new colors,  let $L'(v)=\C\cup \{c_1,c_2\}$, and let $L'(x)=L(x)$ for every $x\in V(G_0)\setminus\{v\}$.

We claim that the canvas $T_1=(G_1',C_1'',L')$ is a critical canvas. 
To see this let $H$ be a proper subgraph of $G_1'$ that includes $C_1''$ as a subgraph. 
Let us extend $\phi_0$ by defining $\phi_0(z_1):=c_1$ and $\phi_0(z_2):=c_2$. 
We will show that (the restriction to $C_1''$ of) $\phi_0$  extends to $H$ but not to $G_1'$. 
If $\phi_0$ extended to $G_1'$, then $\phi_0(v)\not\in\C$ by the definition of $\C$ and
$\phi_0(v)\not\in\{c_1,c_2\}$,  because $v$ is adjacent to $z_1,z_2$, a contradiction.
Thus $\phi_0$ does not extend to $G_1'$.
To show that $\phi_0$  extends to $H$ assume first that $H\setminus\{z_1,z_2\}$ is a proper subgraph of 
$G_0\langle C_1'\rangle$.
Then $(H\setminus\{z_1,z_2\})\cup G_0\langle C_2'\rangle$ is a proper subgraph of $G_0$,
and hence $\phi_0$ extends to it, as desired.
So we may assume that $H\setminus\{z_1,z_2\}=G_0\langle C_1'\rangle$. Since $H$  is
a proper subgraph of $G_1'$  we may assume from the symmetry that $vz_1\not\in E(H)$. 
Now $\phi_0$ extends to an $L'$-coloring of $G_0\setminus v$. Letting $\phi_0(v)=c_1$ shows that $\phi_0$ extends to $H$,
as desired.
This proves the claim that $T_1$ is critical.

As $v(T_1)\ge 2$, we find that $d(T_1)\ge 3-\gamma$ by the minimality of $T$. Also by the minimality of $T$, as $v(T_0'\langle C_2' \rangle)\ge 2$, $d(T_0'\langle C_2' \rangle)\ge 3-\gamma$. 
Let us now count deficiencies. By Lemma~\ref{defsum}, $${\rm def}(T_0')= {\rm def}(T_0'[G'])+{\rm def}(T_0'\langle C_1'\rangle)+{\rm def}(T_0'\langle C_2'\rangle) = -5 + {\rm def}(T_0'\langle C_1'\rangle)+{\rm def}(T_0'\langle C_2'\rangle).$$ Yet, ${\rm def}(T_0)={\rm def}(T_0')+2$. Furthermore, ${\rm def}(T_0' \langle C_1' \rangle)={\rm def}(T_1)+1$. Hence, $${\rm def}(T_0)={\rm def}(T_0'\langle C_1' \rangle)+{\rm def}(T_0'\langle C_2' \rangle)-3={\rm def}(T_1)+{\rm def}(T_0'\langle C_2' \rangle)-2.$$

Next we count the function $s$. We claim that $s(T_0)\le s(T_1)+s(T_0'\langle C_2' \rangle)$. This follows as every vertex of $V(G_0)\setminus V(C_0)$ is either in $V(G_1')\setminus V(C_1'')$ or $V(G_0'\langle C_2'\rangle) \setminus V(C_2')$. Moreover every vertex of $B(T_0)$ is either in $B(T_1)$ or $B(T_0'\langle C_2' \rangle)$ and similarly every vertex of $Q(T_0)$ is either in $Q(T_1)$ or $Q(T_0'\langle C_2' \rangle)$.

Finally putting it all together, we find that 
$$d(T_0) \ge d(T_1)+d(T_0'\langle C_2' \rangle)-2\ge 2(3-\gamma)-2 = 4-2\gamma,$$
as desired.
\end{proof}

\subsection{Tripods}

\begin{definition}
Let $G_0$  be a plane graph with outer cycle $C_0$, and let $v\in V(G_0)\setminus V(C_0)$
have at least three neighbors in $C_0$.
Let $u_1,u_2, \ldots,u_k$  be all the neighbors of $v$ in $C_0$ listed in their order of appearance on $C_0$.
Assume that at most one face of $G_0[V(C_0)\cup \{v\}]$ includes an edge or vertex of $G_0$,
and if such a face exists, then it is incident with $u_1$ and $u_k$.
If $k=3$, then we say that $v$ is {\em tripod of $G_0$}, and
if  $k\ge4$, then we say that $v$ is {\em quadpod of $G_0$}.
The tripod or quadpod $v$ is {\em regular} if there exists a face of $G_0[V(C_0)\cup \{v\}]$ that
 includes an edge or vertex of $G_0$.
If  such a face exists, then we say that $u_1,u_2, \ldots,u_k$ are listed in a {\em standard order}.
Note that every tripod of degree at least four is regular.


If $v$ is a regular tripod or quadpod, we let $C_0\xx{v}$ denote the boundary of the face  of $G_0[V(C_0)\cup \{v\}]$
 that includes an edge or vertex of $G_0$, and we define $G_0\xx{v}:=G_0\langle C_0\xx{v}\rangle$.
If $X$ is a set of tripods or quadpods of $G_0$ and there exists a face  of $G_0[V(C_0)\cup X]$ that includes 
an edge or vertex of $G_0$,
then we let $C_0\xx{X}$ denote the boundary of  such face and we define $G_0\xx{X}:=G_0\langle C_0\xx{X}\rangle$.

Now if $T_0=(G_0,C_0,L_0)$ is a canvas, then we extend all the above terminology to $T_0$
in the natural way. 
Thus we can speak of tripods or quadpods of $T_0$, we define $T_0\xx{X}:=T_0[G_0\xx{X}]$, etc.
\end{definition}


\begin{claim}\label{TripodOrDividing}
Let $T_0=(G_0,C_0,L_0)$ be a canvas with $v(T_0)\ge 2$ and let $v\in V(G_0)\setminus V(C_0)$
have at least three neighbors in $C_0$.
Then $v$ is either a regular tripod of $T_0$, or a true dividing vertex of $T_0$.
\end{claim}
\begin{proof}
Let $u_1,u_2, \ldots,u_k$  be all the neighbors of $v$ in $C_0$ listed in their order of appearance on $C_0$
and numbered such that the face $f$ of $G_0[V(C_0)\cup \{v\}]$  incident with $u_1$ and $u_k$
includes a vertex of $G_0$.
If another face of $G_0[V(C_0)\cup \{v\}]$ includes an edge or vertex of $G_0$, or $k\ge4$,
then by considering the vertices $u_1$ and $u_k$ we find that $v_0$ is a true dividing vertex of $T_0$.
Thus we may assume that $k=3$ and that $f$ is the only  face  of $G_0[V(C_0)\cup \{v\}]$ that includes a vertex 
or edge of $G_0$. It follows that $v$ is a tripod, as desired.
\end{proof}

\begin{definition}
Let $T_0=(G_0,C_0,L_0)$ be a canvas. 
 We say that $T_0$ is a \emph{$0$-relaxation} of $T_0$. Let $k>0$ be an integer,
$T_0'$ be a $(k-1)$-relaxation of $T_0$ and $v$ be a regular tripod of $T_0'$.
Then we say that $T_0'\xx{v}$ is a \emph{$k$-relaxation} of $T_0$. 
\end{definition}

Let us make a few remarks. 
If $T_0$ is a critical canvas, then every tripod  of $T_0$ is regular by Proposition~\ref{Facts}(2).
Therefore $T_0\xx{v}$ is well-defined; moreover, it is a critical canvas by Corollary~\ref{SubCycle}, and $v(T_0\xx{v})\ge2$
again by Proposition~\ref{Facts}(2).
It follows that for all $k\ge 1$, a $k$-relaxation of a critical canvas is well-defined, and if we denote it by $T_0'$, 
then $T_0'$ is critical and $v(T_0')\ge 2$.
Here are some useful claims about relaxations.

\begin{claim}\label{Relax1}
If $T_0'$ is a $k$-relaxation of a canvas $T_0$, then $d(T_0)\ge d(T_0') - k(2\alpha+\epsilon)$.
\end{claim}
\begin{proof}
We proceed by induction on $k$. 
The claim clearly holds for $k=0$, and so we may assume that $k\ge1$ and that the claim holds for all integers
strictly is smaller than~$k$.
Let $T_{k-1}$ be a $(k-1)$-relaxation of $T_0$ and $v$ a regular tripod of $T_{k-1}$ such that $T_0'$ is a $1$-relaxation of $T_{k-1}$. By induction, $d(T_0)\ge d(T_{k-1}) - (k-1)(2\alpha+\epsilon)$. 
Yet ${\rm def}(T_{k-1})={\rm def}(T_0')$ while $v(T_{k-1})=v(T_0')+1$, 
$b(T_{k-1})\le b(T_0')+1$ and $q(T_{k-1})\le q(T_0')+1$. Thus $d(T_{k-1})\ge d(T_0') -(2\alpha+\epsilon)$ and the claim follows.
\end{proof}

\begin{claim}\label{Relax2}
Let $k\in\{0,1,2\}$ and let $T'$ be a $k$-relaxation of $T$.
Then $T'$ does not have a true  dividing vertex, and if $k\le1$, then it does not have a strong dividing vertex.
\end{claim}
\begin{proof}
Suppose for a contradiction that $T'$ has a  true or strong dividing vertex. By Claim~\ref{SubCycle}, $T'$ is critical,
and $v(T')\ge3$ by Claim~\ref{Base}. If $T'$ has a  true dividing vertex, then
$$d(T)\ge d(T')-2(2\alpha+\epsilon)\ge 4-(6\alpha+\gamma+3\epsilon)\ge3-\gamma$$
 by  Claim~\ref{Relax1}, Claim~\ref{DividingTrue}(2)  and inequalities (I2) and (I3).
If $k\le1$ and $T'$ has a strong dividing vertex, then
$$d(T)\ge d(T')-(2\alpha+\epsilon)\ge 4-(2\alpha+2\gamma+\epsilon)\ge3-\gamma,$$
using Claim~\ref{DividingStrong} instead of Claim~\ref{DividingTrue}(2), as desired.
%
%
\end{proof}


\begin{claim}\label{QuasiSame}
If $x_1$ is a tripod of $T$, then letting $T'=T\xx{x_1}$, either
\begin{enumerate}
\item[{\rm(1)}] $\deg(x_1)=5$, or
\item[{\rm(2)}]  $\deg(x_1)=6$, the neighbors of $x_1$ not in $C$ form a path of length two and the ends of that path are in $B(T)$, $b(T)=b(T')$, $q(T)=q(T')$ and $d(T)\ge d(T')-\epsilon$.
\end{enumerate}
\end{claim}
\begin{proof}
Note that as $v(T)\ge 4$ by Claim~\ref{Base}, then $v(T')\ge 3$. As $G$ is $C$-critical, $\deg(x_1)\ge 5$. If $\deg(x_1)=5$, then (1) holds, as desired. So we may assume that $\deg(x_1)\ge 6$. By the minimality of $T$, $d(T')\ge 3-\gamma$. Moreover, ${\rm def}(T)={\rm def}(T')$ and $v(T)=v(T')+1$. Thus $d(T) = d(T') - \epsilon + \alpha(b(T')-b(T)+q(T')-q(T))$. 
Let $c_1,c_2,c_3$ be the neighbors of $x_1$ in $C$ listed in standard order, and let
$c_1,c_2,c_3,q_1, \ldots, q_2$ be all the neighbors of $x_1$ listed in their cyclic order around $x_1$.

Let $R = N(x_1)\setminus \{c_1,c_2,c_3,q_1,q_2\}$. We claim that $R \cap Q(T) = \emptyset$. 
Suppose not, and let $q\in R\cap Q(T)$. Then $q$ is a dividing vertex of $T'$. Given the presence of $q_1$ and $q_2$, $q$ is a strong dividing vertex of $T'$, contrary to Claim~\ref{Relax2}.
This proves  that $R \cap Q(T) = \emptyset$ and implies that $R\cap B(T)=\emptyset$ as well.

Note that $R\subseteq B(T')\subseteq Q(T')$. Thus $q(T') \ge q(T) + |R|-1$ and $b(T')\ge b(T) + |R|-1$. Hence if $|R|\ge 2$, then $d(T) \ge d(T')-\epsilon + 2\alpha \ge d(T')\ge 3-\gamma$ since $2\alpha\ge \epsilon$ by inequality (I1), a contradiction. So $|R|=1$ and $\deg(x_1)=6$. Thus $q(T')\ge q(T)$ and $b(T')\ge b(T)$. Now it follows that $q(T)=q(T')$ and $b(T)=b(T')$ as otherwise $d(T)\ge d(T') - \epsilon + \alpha \ge 3-\gamma$ since $\alpha \ge \epsilon$ by inequality (I1), a contradiction. Hence $d(T)\ge d(T')-\epsilon$.

Let $q\in R$. The conclusions above imply that $Q(T')\setminus \{q\} = Q(T)\setminus \{x_1\}$ and $B(T')\setminus \{q\} = B(T)\setminus \{x_1\}$. The latter implies that $q_1,q_2 \in B(T)$. The former implies that $q_1qq_2$ form a path, for otherwise there would exist a vertex other than $q,q_1,q_2$ that is cofacial with $x_1$ and therefore belongs to $Q(T')$; yet that vertex must then also belong to $Q(T)$ and so be a strong dividing vertex of $T'$, a contradiction as above. Thus (2) holds as desired.
\end{proof}

\begin{claim}\label{Relax4}
For $k\in\{0,1,2, 3\}$, if $T'$ is a $k$-relaxation of $T$, then there does not exist a proper critical subcanvas of $T'$.
\end{claim}
\begin{proof}
Suppose not. By Claim~\ref{ProperCrit}, $d(T')\ge 4-\gamma$. By Claim~\ref{Relax1}, $d(T)\ge 4-\gamma -3(2\alpha+\epsilon)$, which is at least $3-\gamma$ as $6\alpha+3\epsilon \le 1$ by inequalities (I2) and (I3), a contradiction.
\end{proof}
 
%

Let $X_1$ be the set of all vertices $v\in V(G)\setminus V(C)$ with
at least three neighbors in $C$.

\begin{claim}\label{X1}
$X_1\ne \emptyset$ and every member of $X_1$ is a tripod of $T$.
\end{claim}
\begin{proof}
By Claim~\ref{Chord}, there does not exist a chord of $C$, and hence $X_1\ne\emptyset$ by 
Theorem~\ref{CycleChordTripod}.
By Claims~\ref{TripodOrDividing} and~\ref{Relax2} every member of $X_1$ is a tripod of $T$.
\end{proof}

%
%

\begin{claim}\label{TX1welldef}
$T\xx{X_1}$ is well-defined and is a critical canvas. 
\end{claim}

\begin{proof}
By Proposition~\ref{Facts}(2) every tripod of $G$ is  regular, and hence $T\xx{X_1}$ is well-defined.
It is critical by Corollary~\ref{SubCycle}.
\end{proof}

\begin{claim}\label{X1Chord}
The graph $G\xx{X_1}$ does not have a chord of $C\xx{X_1}$.
\end{claim}
\begin{proof}
Suppose not. Let $v_1v_2$ be a chord of $C\xx{X_1}$. As $C$ has no chord by Claim~\ref{Chord}, we may assume without loss of generality that $v_1\not\in V(C)$. Thus $v_1$ is a tripod of $C$. Hence $v_2$ is also a tripod, as otherwise $v_1$ is not a tripod. But then $v_2$ is a true dividing vertex for $T\xx{v_1}$ because $v(T)\ge 4$ by Claim~\ref{Base}, contradicting Claim~\ref{Relax2}.
\end{proof}

\begin{claim}\label{vTX1}
$v(T\xx{X_1})\ge2$. 
\end{claim}

\begin{proof}
Every tripod of $T$ has at least two neighbors in $(G\xx{X_1})\setminus V(C)$  by Proposition~\ref{Facts}(2),
and no   neighbor in $X_1$ by Claim~\ref{X1Chord}.
\end{proof}


Let $X_2$ be the set of all tripods and quadpods of $G\xx{X_1}$.

\begin{claim}\label{X2}
We have $X_2\ne \emptyset$. Furthermore, let $x_2\in X_2$, and let $u_1,u_2\ldots,u_k$ be all the neighbors 
of $x_2$ in $C\xx{X_1}$ listed in  a standard order. 
Then $k=3$ and $u_2\in V(C)$. 
In particular, every member of $X_2$ is a tripod.
\end{claim}

\begin{proof}
By Claim~\ref{X1Chord}, there does not exist a chord of $C\xx{X_1}$, and hence from
Claim~\ref{TX1welldef} and Theorem~\ref{CycleChordTripod} it follows that $X_2\ne \emptyset$. 
Let $x_2\in X_2$ and $u_1,u_2\ldots,u_k$ be as stated.
Let $i\in\{2,3,\ldots,k-1\}$. If $u_i\in X_1$, then $u_i$ has three neighbors in $C$ and is adjacent to $x_2$,
but  has no other neighbors, contrary to Proposition~\ref{Facts}(2).
Thus $u_i\in V(C)$. 

We may assume that $k\ge4$, for otherwise the remaining two assertions hold.
Since $x_2\not\in X_1$ we may assume from the symmetry that $u_1\in X_1$.
By considering the vertices $u_1$ and $u_4$ we find that $x_2$
is a true dividing vertex of either $T\xx{u_1}$ (if $u_4\in V(C)$) 
 or $T\xx{\{u_1,u_4\}}$ (if $u_4\not\in V(C)$), in either case contrary to Claim~\ref{Relax2}.
%
\end{proof}



Since $T$ is a critical canvas there exists an  $L$-coloring of $C$ that does not extend to
an $L$-coloring of $G$. For the rest of the proof let us fix one such $L$-coloring $\phi$.

\begin{claim}\label{c:phiext}
The coloring $\phi$ extends to every proper subgraph of $G$ that contains $C$ as a subgraph.
\end{claim}
\begin{proof}
This follows from  Proposition~\ref{p:critsubcanvas} and Claim~\ref{NoProperCrit}.
\end{proof}




For $v\in V(G)\setminus V(C)$ we let $S(v):=L(v)\setminus \{\phi(u)\,|\, u\in N(v)\cap V(C)\}$.

\begin{claim}\label{LS}
For all $v\in V(G)\setminus V(C)$, $|L(v)|=5$ and $|S(v)|=5-|N(v)\cap V(C)|$.
\end{claim}
\begin{proof}
Suppose for a contradiction that $|L(v)|\ge6$ for some $v\in V(G)\setminus V(C)$.
Let $c\in L(v)$ and let $L'$ be defined by $L'(v):=L(v)\setminus\{c\}$ and $L'(x):=L(x)$
for all $x\in V(G)\setminus\{v\}$.
Then $(G,C,L')$ is a canvas and $\phi$ clearly does not extend to an $L'$-coloring of $G$.
By  Proposition~\ref{p:critsubcanvas} the canvas $(G,C,L')$ has a critical subcanvas $(G',C,L')$.
Condition (M2) in the choice of $T$  implies that $G'$ is a proper  subgraph of $G$,
but that contradicts Claim~\ref{ProperCrit} applied to $T_0=T$ and $G'$.
This proves that $|L(v)|=5$ for every $v\in V(G)\setminus V(C)$.

To prove the second statement suppose for a contradiction that $|S(v)|>5-|N(v)\cap V(C)|$ 
for some $v\in V(G)\setminus V(C)$.
Thus $v$ has two distinct neighbors $w_1,w_2\in V(C)$ such that $\phi(w_1)=\phi(w_2)$.
But then $\phi$ does not extend to $G \setminus vw_1$, contrary to Claim~\ref{c:phiext}.
\end{proof}


By Claim~\ref{X2}  there exists $x_2\in X_2$. Let $u_1,u_2,u_3$ be as in Claim~\ref{X2},
and let $U:=\{u_i\,|\,u_i\in X_1\}$. Thus $U\ne\emptyset$ and $U\subseteq\{u_1,u_3\}$.
Let us choose $x_2$ such that $|U|$ is minimized.
We refer to Figures~\ref{fig:tripod1} and~\ref{fig:tripod2} for a depiction of $x_2$
and two other vertices whose existence will be established shortly.

\begin{figure}[htb]
 \centering
\includegraphics[scale = .25]{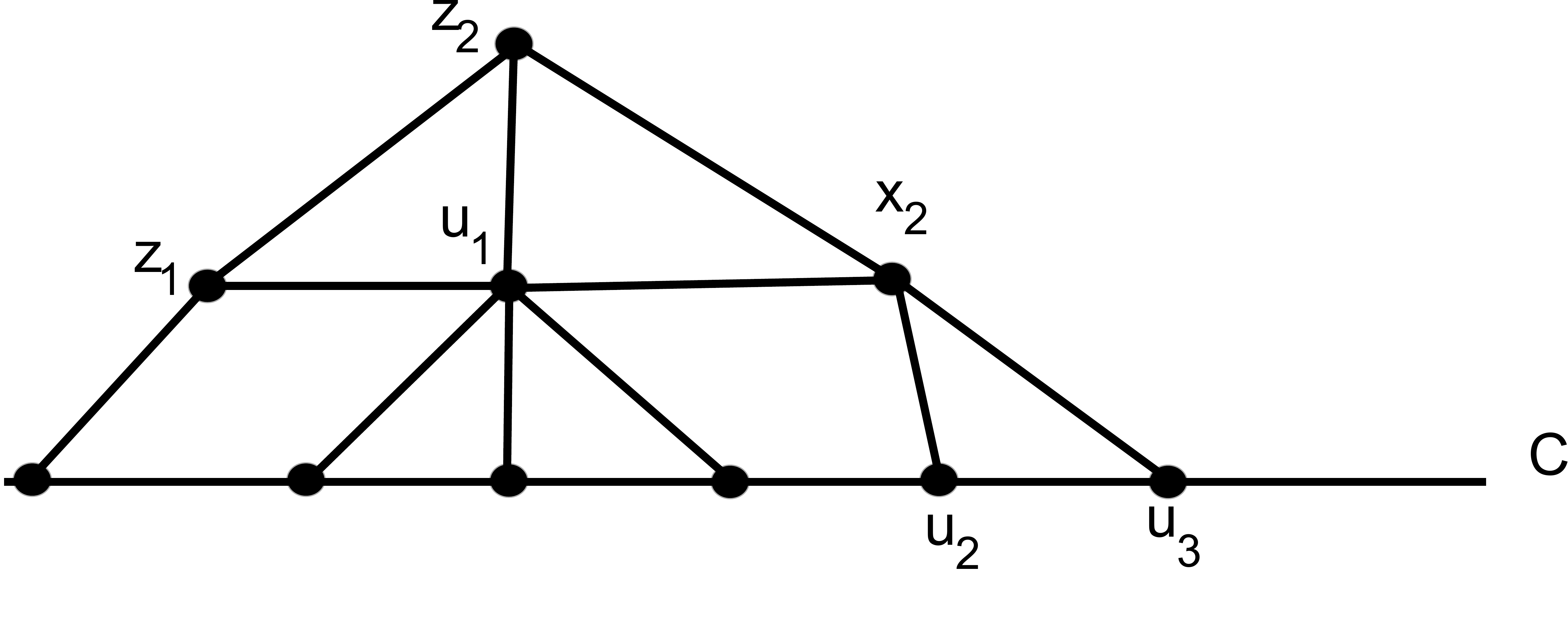}
 \caption{Case 1 of a tripod of $G\xx{X_1}$.}
\label{fig:tripod1}
\end{figure}

\begin{figure}[htb]
 \centering
\includegraphics[scale = .25]{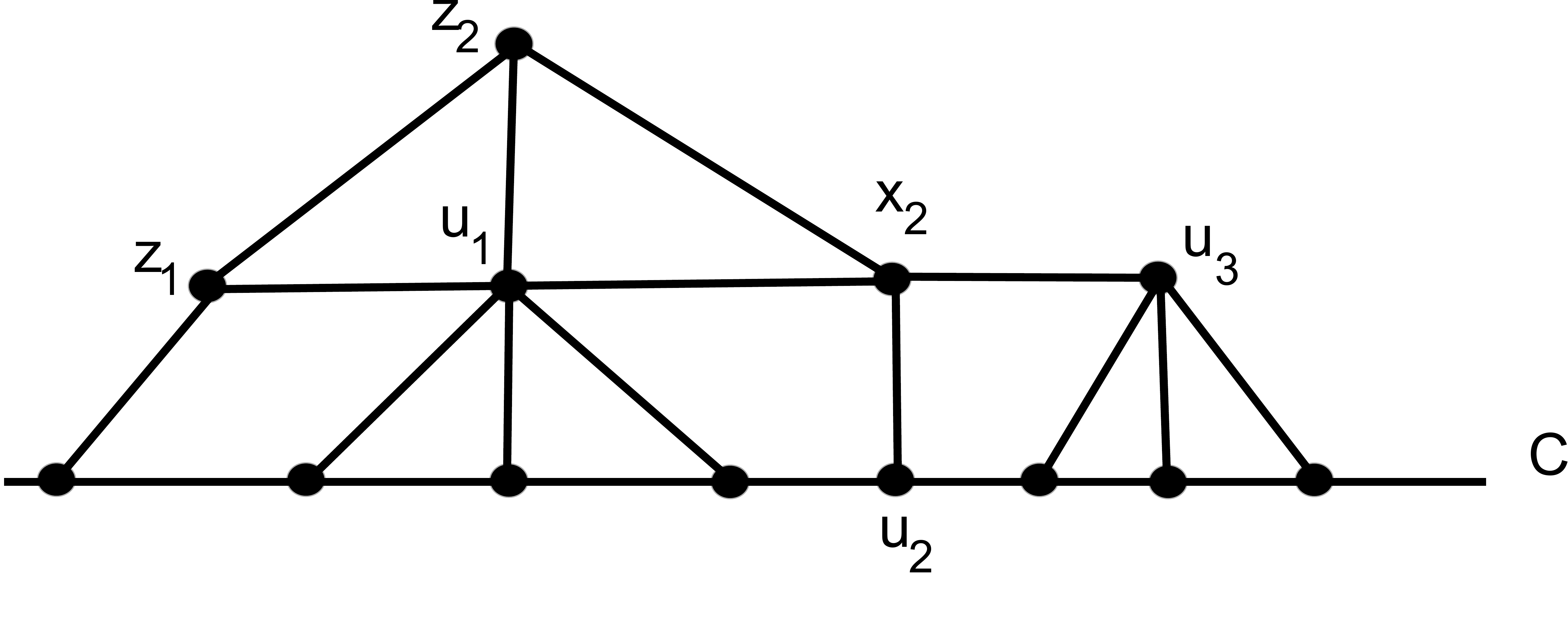}
 \caption{Case 2 of a tripod of $G\xx{X_1}$.}
\label{fig:tripod2}
\end{figure}


\begin{claim}\label{DegSpecial}
If $u\in U$, then $\deg(u)=6$ and there exist adjacent vertices $z_1,z_2\not\in V(C)$ such that $z_1$ is adjacent to $u$ and is in $B(T)$, and $z_2$ is adjacent to $u$ and $x_2$.
\end{claim}
\begin{proof}
By Claim~\ref{QuasiSame} applied to $T$ and $u$, we find that $\deg(u)\le 6$ and the claim follows unless $\deg(u)=5$. So suppose  for a contradiction that $\deg(u)=5$.

Let $C'=(C\xx{U})\xx{x_2}$ and $T\langle C'\rangle =(G',C',L)$. 
Let $z\in V(G')\setminus V(C')$ be a neighbor of $u$.
We claim that $G'\setminus uz$ has a $C'$-critical subgraph. 
To see this we extend $\phi$ to an $L$-coloring $\phi'$ of $C\cup C'$ as follows.
For $v\in V(C)$ let $\phi'(v):=\phi(v)$.
Since $x_2\not\in X_1$ we have $|S(x_2)|\ge 3$, and $|S(u)|=2$ by Claim~\ref{LS}.
We may  therefore choose $\phi'(x_2)\in S(x_2)\setminus S(u)$.
For $u'\in U$ select $\phi'(u')\in S(u')\setminus\{ \phi(x_2)\}$.
Now if $\phi'$ extends to an $L$-coloring $\phi''$ of $G'\setminus uz$, then 
by re-defining $\phi''(u)$  to be a color in $S(u)\setminus\{\phi(z)\}$ we obtain an extension
of  $\phi$ to an $L$-coloring of $G$, a contradiction.
Thus $\phi'$ does not extend to an $L$-coloring  of $G'\setminus uz$, and so by
Proposition~\ref{p:critsubcanvas}
this proves our claim that $G'\setminus uz$ has a $C'$-critical subgraph, say $G''$. 
But  $G''$ is a proper $C'$-critical subgraph of $G'$, contradicting Claim~\ref{Relax4}. 
\end{proof}

%

\begin{claim}\label{TripodS}
If $u\in U$ and $z\in N(u)\setminus V(C)$, then $S(u)\subseteq S(z)$.
\end{claim}
\begin{proof}
Suppose that $S(u)\setminus S(z)\ne \emptyset$. 
Let $C'=C\xx{u}$ and $T\langle C'\rangle=(G',C',L)$. 
We claim that $G'\setminus uz$ has a $C'$-critical subgraph. To see this we extend $\phi$ to an $L$-coloring $\phi'$ of $C\cup C'$ as follows.
For $v\in V(C)$ let $\phi'(v):=\phi(v)$, and we choose $\phi'(u)\in S(u)\setminus S(z)$. 
Now  $\phi'$ does not extend to an $L$-coloring  of $G'$, for such an extension would be 
an $L$-coloring  of $G$, a contradiction.
By Proposition~\ref{p:critsubcanvas}
this proves our claim that $G'\setminus uz$ has a $C'$-critical subgraph, say $G''$. 
But  $G''$ is a proper $C'$-critical subgraph of $G'$, contradicting Claim~\ref{Relax4}. 
%
\end{proof}

\begin{claim}\label{BipodS}
If $z\not\in V(C\xx{U})$ is a neighbor of $x_2$ in $G\xx{U}$, then $S(x_2)\subseteq S(z)$.
\end{claim}
\begin{proof}
This follows in a similar manner as the proof of  Claim~\ref{DegSpecial}. 
We extend $\phi$ to an $L$-coloring $\phi'$ of $C\cup C'$ as follows.
For $v\in V(C)$ let $\phi'(v):=\phi(v)$.
We  choose $\phi'(x_2)\in S(x_2)\setminus S(z)$, and
for $u'\in U$ we select $\phi'(u')\in S(u')\setminus\{ \phi(x_2)\}$.
An  extension of $\phi'$ to an $L$-coloring of $G'$ would be an $L$-coloring of $G$.
The rest of the argument is identical to the proof of Claim~\ref{DegSpecial}.
\end{proof}

\begin{claim}\label{Z1}
If $u\in U$ and $z_1, z_2$ are as in Claim~\ref{DegSpecial}, then $N(z_2)\cap V(C\xx{U})=\{u\}$, 
and $|N(z_1)\cap V(C\xx{U})|\le 3$. Furthermore if $|U|=2$, then $|N(z_1)\cap V(C\xx{U})|\le 2$.
\end{claim}
\begin{proof}
Note that $z_1$ and $z_2$ are adjacent to $u$. Suppose that $|N(z_2)\cap V(C\xx{U})|\ge 2$. 
But then as $\deg(u)=6$, $z_2$ is a true dividing vertex of $G\xx{U}$, a contradiction by Claim~\ref{Relax2}.
Since $G$ has no separating $4$-cycles by Proposition~\ref{Facts}, $N(z_1)\cap U = \{u\}$. 
Since $z_1\not\in X_1$ as $C\xx{X_1}$ has no chords by Claim~\ref{X1Chord}, it follows that $|N(z_1)\cap V(C)|\le 2$,
and hence $|N(z_1)\cap V(C\xx{U})|\le 3$. 
Furthermore suppose $|U|=2$ and $|N(z_1)\cap V(C\xx{U})|\ge 3$. 
Then $|N(z_1)\cap V(C)|=2$; moreover, as $z_1$ is not a true dividing vertex of $G\xx{u}$
by Claim~\ref{Relax2}, it follows from Claim~\ref{TripodOrDividing} that $z_1\in X_2$. 
But then $z_1$ contradicts the choice of $x_2$. 
\end{proof}

Let $T_1:=T\langle C\xx{U}\rangle$ and $T_2:=T_1\xx{x_2}$. 

\begin{claim}\label{dU}
$d(T)\ge d(T_1)-|U|\epsilon$.
\end{claim}
\begin{proof}
This follows by showing that $b(T_1)\ge b(T)$ and $q(T_1)\ge q(T)$. To see that, note that by Claim~\ref{DegSpecial}, $\deg(u)=6$ for all $u\in U$. 
Then by Claim~\ref{QuasiSame}, $b(T\xx{u})=b(T)$ and $q(T\xx{u})=q(T)$ for all $u\in U$. 
Thus if $b(T_1) < b(T)$ or $q(T_1) < q(T)$, it must be that $|U|=2$. 
Furthermore, then  the two vertices in $U$ either have a common neighbor $z$ or common cofacial vertex $z$. 
In either case, $z$ is a strong dividing vertex of $C\xx{U}$.
 Hence by Claim~\ref{DividingStrong}, $d(T_1) \ge 4-2\gamma$. 
Yet $b(T_1) \ge b(T)-1$ and $q(T_1)\ge q(T)-1$. Thus $d(T)\ge d(T_1)-2\epsilon-2\alpha$. So $d(T)\ge 4-2\gamma-2\alpha-2\epsilon$ which is at least $3-\gamma$ as $2\alpha + 2\epsilon + \gamma \le 1$ by inequality (I3), a contradiction.
\end{proof}

\begin{claim}\label{degX2}
$\deg(x_2)=6$ and $d(T)\ge d(T_2)-(|U|+1)\epsilon$.
\end{claim}
\begin{proof}
Suppose not. Suppose $\deg(x_2)\ge 6$. 
As $x_2$ is a tripod of $T_1$ by Claim~\ref{X2}, $v(T_1)\ge 3$ and hence $v(T_2)\ge 2$. Thus by the minimality of $T$, $d(T_2)\ge 3-\gamma$. Moreover, ${\rm def}(T_1)={\rm def}(T_2)$ and $v(T_1)=v(T_2)+1$. Thus $d(T_1) = d(T_2) - \epsilon + \alpha(b(T_2)-b(T_1)+q(T_2)-q(T_1))$. 
Let $u_1,u_2,u_3$ be as Claim~\ref{X2}, and let
$u_1,u_2,u_3,q_1, \ldots,q_2$  be all the neighbors of $x_2$ listed in their cyclic order around $x_2$.


Let $R = N(x_2)\setminus \{u_1,u_2,u_3,q_1,q_2\}$. We claim that $R \cap Q(T_1) = \emptyset$. 
Suppose not, and let $q\in R\cap Q(T_1)$. Then $q$ is a dividing vertex of $T_2$. 
Given the presence of $q_1$ and $q_2$, $q$ is a strong dividing vertex of $T_2$. 
By Claim~\ref{DividingStrong}, $d(T_2)\ge 4-2\gamma$. 
As $d(T_1)\ge d(T_2) - (2\alpha+\epsilon)$, $d(T_1)\ge 4-2\gamma-(2\alpha+\epsilon)$. Hence $d(T)\ge 4-2\gamma-(2\alpha+\epsilon)-2\epsilon$, a contradiction as $2\alpha+3\epsilon\ + \gamma\le 1$ by inequality (I3). 
This proves the claim. Note that this implies that $R\cap B(T_2)=\emptyset$ as well.

Note that $R\subseteq B(T_2) \subseteq Q(T_2)$. Thus $q(T_2) \ge q(T_1) + |R|-1$ and $b(T_2)\ge b(T_1) + |R|-1$. Suppose $\deg(x_2)\ge 7$, then $|R|\ge 2$. Thus $d(T_1) \ge d(T_2)-\epsilon + 2\alpha$ and $d(T)\ge d(T_2)-3\epsilon +2\alpha$. As $d(T_2)\ge 3-\gamma$ by the minimality of $T$, $d(T)\ge 3-\gamma$ since $2\alpha\ge 3\epsilon$ by inequality (I1), a contradiction. So suppose $\deg(x_2)=6$. Then $|R|=1$ and hence 
$q(T_2)\ge q(T_1)$ and $b(T_2)\ge b(T_1)$. It follows that $d(T_1) \ge d(T_2)-\epsilon$ and hence $d(T)\ge d(T_2)-(|U|+1)\epsilon$ and the claim holds as desired.

So we may assume that $\deg(x_2)=5$. Let $u\in U$ and $z_1,z_2$ be as in Claim~\ref{DegSpecial}. 
By Claim~\ref{TripodS}, $S(u)\subset S(x_2)$ and hence $L(z_2)\setminus (S(u)\cup S(x_2)) = L(z_2)\setminus S(x_2)$.
 By Claim~\ref{Z1}, $N(z_2)\cap V(C)=\emptyset$ and thus $|L(z_2)\setminus S(x_2)|\ge 1$ as $|L(z_2)|=5$ and $|S(x_2)|\le 4$ by Claim~\ref{LS}. 
Let $C'$ be obtained from $(C\xx{U})\xx{x_2}\setminus \{ux_2\}$ by adding the 
vertex $z_2$ and edges $uz_2,x_2z_2$, and let $T'=(G',C',L)=T\langle C'\rangle$. Note $T'$ is critical by Corollary~\ref{SubCycle}. 

Consider $G'\setminus \{uz_1, x_2z_3\}$, where $z_3\in N(x_2)\setminus (V(C')\cup V(C))$. We claim that $G'\setminus \{uz_1,x_2z_3\}$ has a $C'$-critical subgraph. To see this, choose $\phi(z_2)\in L(z_2)\setminus S(x_2)$. Also choose $\phi(u')\in S(u')$ if $|U|=2$ where $u' \in U\setminus \{u\}$. If $\phi$ extends to an $L$-coloring of $G'\setminus \{u, x_2\}$, then $\phi$ could be extended to an $L$-coloring of $G$ as follows. 
First extend $\phi$ to $u$ by choosing $\phi(u)\in S(u)\setminus \phi(z_1)$ which is non-empty as $|S(u)|=2$. 
Then extend $\phi$ to $x_2$ by choosing $\phi(x_2)\in S(x_2)\setminus \{\phi(u),\phi(u'),\phi(z_3)\}$. 
This set is non-empty since if $|U|=1$, then $|S(x_2)|=3$ and if $|U|=2$, then $|S(x_2)|=4$. Hence $\phi$ could be extended to an $L$-coloring of $G$, contradicting that $T$ is a counterexample. Thus $\phi$ does not extend to $G'\setminus \{u,x_2\}$ and so does not extend to $G'\setminus \{uz_1,x_2z_3\}$. 
By Proposition~\ref{p:critsubcanvas} this proves the claim that $G'\setminus \{uz_1,x_2z_3\}$ has a $C'$-critical subgraph.

Thus $G'$ contains a proper $C'$-critical subgraph $G''$.  Note that $v(T')\ge 3$ given that $\deg(z_2)\ge 5$ and $|N(z_2)\cap V(C')|\le 2$. Moreover, $|E(G')\setminus E(G'')|\ge 2$. In addition, we claim that $|E(G'')\setminus E(C')|\ge 2$. Suppose not. Then there would exist a chord of $C'$, which would imply that $z_2$ is adjacent to a vertex in $C$. But then $z_2$ is a true dividing vertex of $C\xx{u}$, contradicting Claim~\ref{Relax2}. This proves the claim that $|E(G'')\setminus E(C')|\ge 2$.

By Claim~\ref{ProperCrit}(3) applied to $T'$ and $G''$, we find that $d(T')\ge 5-\gamma-(2\alpha+\epsilon)$. Moreover, $s(T_1)\le s(T')+2(2\alpha+\epsilon)$, ${\rm def}(T_1)={\rm def}(T')-1$ and hence $d(T_1)\ge d(T')-1-2(2\alpha+\epsilon)$. Thus $d(T_1)\ge 4-\gamma-3(2\alpha+\epsilon)$. Yet $d(T)\ge d(T_1)-2\epsilon$ by Claim~\ref{dU}. So $d(T)\ge 4-\gamma-6\alpha-5\epsilon$ which is at least $3-\gamma$ as $6\alpha+5\epsilon\le 1$ by inequalities (I2) and (I3), a contradiction.
This completes the proof of Claim~\ref{degX2}.
\end{proof}

Let $u\in U$ and $z_1,z_2$ be as in Claim~\ref{DegSpecial}. 
Let $C'$ be obtained from $(C\xx{U})\xx{x_2}\setminus\{u\}$ by adding the vertices $z_1,z_2$ and  edges $yz_1,z_1z_2,z_2x_2$, where $y\in N(z_1)\cap V(C)$ is chosen so that $|V(C')|$ is minimized. Let $T'=(G',C',L)=T\langle C' \rangle$. Consider $G'\setminus \{x_2z_3, x_2z_4\}$,
 where $z_3\ne z_4\in N(x_2)\setminus (V(C\xx{U})\cup V(C)\cup \{z_2\})$. 

We claim that $G'\setminus \{x_2z_3,x_2z_4\}$ has a $C'$-critical subgraph.
To see this
choose $\phi(z_1)\in S(z_1)\setminus S(u)$, which is nonempty as $|S(z_1)|\ge 3$ by Claims~\ref{Z1} and~\ref{LS}. Then choose $\phi(z_2)\in S(z_2)\setminus (S(x_2)\cup\{\phi(z_1)\})$ which is nonempty as $|L(z_1)|=|S(z_1)|=5$ by Claims~\ref{Z1} and ~\ref{LS}. Furthermore if $|U|=2$, then choose $\phi(u')\in S(u')$ where $u'\in U\setminus \{u\}$. If $\phi$ extends to an $L$-coloring of $G'\setminus\{u,x_2\}$, then $\phi$ could be extended to an $L$-coloring of $G$ as follows. First extend $\phi$ to $x_2$ by choosing $\phi(x_2)\in S(x_2)\setminus \{\phi(z_3),\phi(z_4), \phi(u')\}$ which is non-empty since by Claim~\ref{LS}, $|S(x_2)|\ge 3$ if $|U|=1$ and $|S(x_2)|\ge 4$ if $|U|=2$. Then extend $\phi$ to $u$ by choosing $\phi(u)\in S(u)\setminus \phi(x_2)$ which is non-empty as $|S(u)|=2$ by Claim~\ref{LS}. But this contradicts that $T$ is a counterexample.
Thus $\phi$ does not extend  to an $L$-coloring of $G'\setminus\{u,x_2\}$.
By Proposition~\ref{p:critsubcanvas} this proves the claim that $G'\setminus \{x_2z_3,x_2z_4\}$ has a $C'$-critical subgraph.


Thus $G'$ contains a proper $C'$-critical subgraph $G''$.  Note that $z_1$ has at least one neighbor in $G'$ not in $C'$ and $z_1$ is not adjacent to either $z_3$ or $z_4$ as $G$ has no separating $4$-cycles by Proposition~\ref{Facts}. Hence $v(T')\ge 3$. Moreover, $|E(G')\setminus E(G'')|\ge 2$. In addition, we claim that $|E(G'')\setminus E(C')|\ge 2$,
for otherwise there would exist a chord of $C'$, which is impossible given the choice of $C'$ and by Claim~\ref{Z1}.

By Claim~\ref{ProperCrit}(3) applied to $T'$ and $G''$, we find that $d(T')\ge 5-(2\alpha+\epsilon)-\gamma$. Moreover, $s(T_2)\le s(T')+2(2\alpha+\epsilon)$, ${\rm def}(T_2)\ge {\rm def}(T')-1$ and hence $d(T_2)\ge d(T')-1-2(2\alpha+\epsilon)$. Thus $d(T_2)\ge 4-3(2\alpha+\epsilon)-\gamma$. By Claim~\ref{degX2}, $d(T)\ge d(T_2)-3\epsilon \ge 4-\gamma - 6\alpha-6\epsilon$ which is at least $3-\gamma$ as $6\alpha+6\epsilon\le 1$ by inequalities (I2) and (I3), a contradiction.
This concludes the proof of Theorem~\ref{StrongLinear}.

\section{Consequences of Theorem~\ref{StrongLinear}}
\label{sec:conseqstronglin}

Let us state Theorem~\ref{StrongLinear} with explicit constants while omitting boundary and quasi-boundary from the formula.

\begin{theorem}\label{StrongLinear2}
If $(G,C,L)$ is a critical canvas, then $$|V(G)\setminus V(C)|/18 + \sum_{f\in \F(G)}(|f|-3)\le |C|-4.$$
\end{theorem}
\begin{proof}
Let $\epsilon=1/18$, $\alpha=1/12$ and $\gamma=2/3$. By Proposition~\ref{d0}, Theorem~\ref{StrongLinear} and
Lemma~\ref{deffaces}
$$1\le d(G)\le {\rm def}(G)-v(G)/18=|C|-3 - \sum_{f\in \F(G)} (|f|-3)-v(G)/18,$$
and the theorem follows.
\end{proof}

The following corollary follows immediately.

\begin{corollary}\label{CycleBoundedFace} 
%
Let $(G,C,L)$ be a critical canvas. If $f$ is an internal face of $G$, then $|f| < |C|-1$.
\end{corollary}

Omitting face sizes from the formula of Theorem~\ref{StrongLinear2} gives the following even simpler bound:

\begin{theorem}\label{LinearCycle}
If $(G,C,L)$ is a critical canvas, then $|V(G)|\le 19 |V(C)|$.
\end{theorem}
\begin{proof}
Theorem~\ref{StrongLinear2} implies
$|V(G)\setminus V(C)|\le 18|C|$, and hence $|V(G)|= |V(G)\setminus V(C)|+|V(C)|\le 19|V(C)|$.
\end{proof}

We are now ready to prove Theorem~\ref{LinearCycle0}. Indeed, Theorem~\ref{LinearCycle} is a stronger version of Theorem~\ref{LinearCycle0}.
\medskip

\noindent 
{\bf Proof of Theorem~\ref{LinearCycle0}.}
Let $G,C,L$ and $H$ be as in the statement of Theorem~\ref{LinearCycle0}.
%
We claim that $H$ is $C$-critical.
Suppose not. Then, by definition, there exists a proper subgraph $H'$ of $H$ such that for every $L$-coloring $\phi$ of $C$, 
$\phi$ extends to $H'$ if and only if $\phi$ extends to $H$. 
But now if $\phi$ is an $L$-coloring of $C$ that extends to $H'$, then $\phi$ extends to $H$ and hence to $G$,
contradicting the minimality of $H$. 
This proves the claim  that $H$ is $C$-critical.

It follows from Lemma~\ref{2Conn} that $(H,C,L)$ is a critical canvas. 
By Theorem~\ref{LinearCycle} we have $|V(H)|\le 19 |V(C)|$, as desired.\hfill\qed

\section*{Acknowledgment}
The results of this paper form part of the doctoral dissertation~\cite{PosPhD} of the first author,
written under the guidance of the second author.

\baselineskip 11pt
\vfill
\noindent
This material is based upon work supported by the National Science Foundation.
Any opinions, findings, and conclusions or
recommendations expressed in this material are those of the authors and do
not necessarily reflect the views of the National Science Foundation.
\eject

\end{document}